\documentclass[12pt]{amsart}

\numberwithin{equation}{section}

\usepackage{amsmath,amsthm,amsfonts,eucal,}
\usepackage{xypic}
\usepackage{graphicx}
\usepackage{amssymb}

\def\cb{{\mathcal B}}

\def\cf{{\mathcal F}}

\def\ch{{\mathcal H}}

\def\cs{{\mathcal S}}

\def\ga{{\mathfrak A}}

\def\gar{{\mathfrak R}}

\def\bc{{\mathbb C}}

\def\bn{{\mathbb N}}

\def\br{{\mathbb R}}

\def\bz{{\mathbb Z}}

\def\a{\alpha}
\def\b{\beta}

\def\l{\lambda} 

\def\m{\mu}

\def\s{\sigma} 
\def\t{\tau}

\def\om{\omega} \def\Om{\Omega}

\newtheorem{thm}{Theorem}[section]
\newtheorem{example}[thm]{Example}
\newtheorem{lem}[thm]{Lemma}
\newtheorem{cor}[thm]{Corollary}
\newtheorem{prop}[thm]{Proposition}

\theoremstyle{definition}
\newtheorem{rem}[thm]{Remark}

\begin{document}
\title[On truncated $t$-free Fock spaces]
{On truncated $t$-free Fock spaces: spectrum of position operators and shift-invariant states}
\author{Vitonofrio Crismale}
\address{Vitonofrio Crismale\\
Dipartimento di Matematica\\
Universit\`{a} degli studi di Bari\\
Via E. Orabona, 4, 70125 Bari, Italy}
\email{\texttt{vitonofrio.crismale@uniba.it}}
\author{Simone Del Vecchio}
\address{Simone Del Vecchio\\
Dipartimento di Matematica\\
Universit\`{a} degli studi di Bari\\
Via E. Orabona, 4, 70125 Bari, Italy} \email{{\tt
simone.delvecchio@uniba.it}}
\author{Stefano Rossi}
\address{Stefano Rossi\\
Dipartimento di Matematica\\
Universit\`{a} degli studi di Bari\\
Via E. Orabona, 4, 70125 Bari, Italy}
\email{\texttt{stefano.rossi@uniba.it}}
%\date{\today}
\begin{abstract}
The ergodic properties of the shift on  both full and $m$-truncated $t$-free  $C^*$-algebras are analyzed. In particular, the shift is shown to be uniquely ergodic with respect to the fixed-point algebra.
 In addition, for every $m\geq 1$, the invariant states of the shift acting on the $m$-truncated $t$-free  $C^*$-algebra
are shown to yield a $m+1$-dimensional Choquet simplex, which collapses to a segment in the full case.
Finally, the spectrum of the position operators on  the $m$-truncated $t$-free Fock space is also determined.

\end{abstract}

\maketitle
%\tableofcontents
\section{Introduction}

A  class of non-commutative $C^*$-dynamical systems  arise from concrete models coming from
Quantum Probability.  The shift automorphism certainly represents an interesting example, not least because it is closely related to
stationary stochastic processes. What is more relevant to non-commutative ergodic theory  is that
the shift displays different ergodic properties
depending on the 
$C^*$-algebra it acts on.  Little can be said when it acts on the CAR algebra since its shift-invariant states are probably too many to
be characterized in any useful way, see \cite{CFCMP,CrFid, CR}.
On both Boolean and monotone $C^*$-algebras, however,  the shift only features a segment of invariant states, whose extreme points are the vacuum
and the state at infinity, \cite{CFL}. Lastly,  on the $C^*$-algebra
generated by $q$-deformed commutation relations the shift is even better-behaved in that
it acts as a uniquely ergodic automorphism whenever $|q|<1$,  \cite{DyF}.\\
In the present work we show that suitable choices of the $C^*$-algebra that we make the shift act on  
give rise to new ergodic properties, which are in a sense intermediate between those we recollected above.
Since the third example includes in particular the free case, which corresponds to $q=0$, it  is reasonable to
analyze the ergodic properties of the shift when it acts on the $C^*$-algebra generated by
the creators on truncated  free Fock spaces. These spaces are not entirely new in the context of Quantum Probability. In fact,  they
have been  studied by Franz and Lenczewski in \cite{FL}, where they are referred to as the $m$-free Fock spaces, as a means to  obtain what they call a
hierarchy of freeness along with the corresponding central limit theorems.
The spaces we are actually concerned with in this paper are slightly more general than those in \cite{FL} because
they are obtained by truncating the $t$-free Fock space, which are at the same time a generalization of the full Fock space and a particular yet notable case of so-called one-mode type interacting
Fock space, \cite{ACL}. These were defined by Bo\.{z}ejko and Wysoczanski in \cite{BW} to provide a concrete  model 
where the distribution in the vacuum of the position operators  was exactly that of the central limit of the 
$t$-free convolution.  This operation between measures had in turn been introduced  in \cite{BW0} by the same authors
following the influential work of Bo\.{z}ejko, Leinert and Speicher \cite{BLS} on  conditionally free independence.

Let us now move on to present the results of our paper. After defining the truncated $t$-free Fock spaces for each natural $m\geq 1$, we first focus on the spectrum of the corresponding position operators. This turns out to be a finite set strictly larger than the support of the spectral measure associated with the vacuum vector, meaning that the latter fails to be separating for the $C^*$-algebra generated by a single position operator.
In addition, in Proposition \ref{finsupp} we show that the spectral measure is a convex combination of $m+1$ Dirac measures at the zeros of explicitly given polynomials whose weights are also given.\\
We then analyze the ergodic properties of the action of the shift on  the $C^*$-algebras generated by the creators on both the truncated $t$-free Fock spaces and on the full $t$-free Fock space. The corresponding dynamical systems are uniquely ergodic with respect to the fixed-point algebra, a notion introduced in \cite{AD} and further studied in \cite{DFR1, DFR2, F}. Furthermore, for $m\geq 1$, the invariant states on  the $C^*$-algebra relative to the $t$-free Fock space yield a $m+1$-dimensional Choquet simplex which collapses to a segment\footnote{Actually a point when $t=1$.} when any number of particles is allowed, Propositions \ref{simplex} and \ref{fullue}. Finally, in analogy with what happens in the Boolean and monotone cases, we show that the $C^*$-algebra generated by position operators agrees with the $C^*$-algebra generated by creation and annihilator operators and that it acts irreducibly on the truncated $t$-free Fock space.

\label{sec1}
\section{Preliminaries}\label{prel}

Let $\ch$ be a fixed Hilbert space.
For any positive real number $t$ the $t$-free Fock Hilbert space $\cf_t(\ch)$, is the completion of the algebraic direct sum
$$\bc\Om\oplus\ch\oplus\ch^{\otimes^2}\oplus\cdots\oplus\ch^{\otimes^n}\oplus\cdots$$
with respect to the inner product
$$\langle f_1\otimes\cdots \otimes f_n, g_1\otimes\cdots \otimes g_n \rangle=\delta_{n, m}t^{n-1}\prod_{i=1}^n \langle f_i, g_i \rangle$$
for every $f_1, \ldots, f_n$ and $g_1, \ldots, g_m$ in $\ch$ and $n, m\geq 1$, where $\Om$ is the vacuum vector.
We would like to mention that the choice $t=1$ returns the free Fock space, and that this Hilbert space can also be viewed as a one-mode type interacting
Fock space with $\lambda_n= t^{n-1}$; for details the reader is referred to \cite{AB, ACL}.
In the sequel, we will always take $\ch=\ell^2(\bz)$ with canonical orthonormal basis $\{e_i: i\in\bz\}$.
Creators and annihilators on the $t$-free Fock Hilbert space are defined as follows. For each $k\in\bn$ and  $i,j_1,\ldots, j_k \in \bz$
\begin{align}
\begin{split}\label{creator}
&a^\dag_i\Om:= e_i\\
&a^\dag_i  e_{j_1}\otimes \cdots \otimes e_{j_k}:=e_i\otimes e_{j_1}\otimes \cdots \otimes e_{j_k}
\end{split}
\end{align}

\begin{align}
\begin{split}\label{ann}
&a_i\Om:= 0\,, \\
&a_ie_{j_1}\otimes \cdots \otimes e_{j_k}:=\left \{
\begin {array} {ll}
\delta_{i, j_1}\Om & \text{if $k=1$}\\
t \delta_{i, j_1}e_{j_2}\otimes \cdots \otimes e_{j_k}\ &\text{if $k\geq 2$}\\
\end{array}
\right.
\end{split}
\end{align}

For each $i\in\bz$ the position operators $x_i$ are the self-adjoint part of the annihilators, namley $x_i=a_i+a^\dag_i$.
We denote by $\ga_t\subset\cb(\cf_t(\ch))$ the unital  $C^*$-algebra generated by all creators.\\

For any fixed  $m\geq 1$ we introduce what we call the $m$-truncated $t$-free Fock space as 
$$\cf_t^{(m)}(\ch):=\bc\Om\oplus\ch\oplus\ch^{\otimes^2}\oplus\cdots\oplus\ch^{\otimes^m}\subset \cf_t(\ch)\,.$$
We denote by $P_m$ the orthogonal projection of $\cf_t(\ch)$ onto $\cf_t^{(m)}(\ch)$.
It is worth noting that, for each $m\geq 1$, the Hilbert space $\cf_t^{(m)}(\ch)$ is again a
one-mode interacting Fock space with $\lambda_0=1$, $\lambda_n=t^{n-1}$ for $n=1, 2, \ldots, m$
and $\lambda_n=0$ for all $n\geq m+1$.\\
The corresponding $m$-truncated annihilation and position operators are defined as  $a^{(m)}_i:=P_ma_iP_m$ and $x^{(m)}_i:=P_ma_iP_m$. We also consider the $C^*$-algebras generated by the truncated operators defined above:
$$\ga^{(m)}_t:=C^*(\{a^{(m)}_i: i\in \bz\})\subset\cb(\cf_t^{(m)}(\ch))\,,$$
$$\gar^{(m)}_t:=C^*(\{x^{(m)}_i: i\in \bz\})\subset\cb(\cf_t^{(m)}(\ch))\,.$$
These $C^*$-algebras  will be every so often referred to as $m$-truncated $t$-free  $C^*$-algebras.
Note  both $\ga^{(m)}_t$ and $\gar^{(m)}_t$
 are acted upon by $*$-automorphism $\t$ sending $a^{(m)}_i$ in $a^{(m)}_{i+1}$, and $x^{(m)}_i$ in $x^{(m)}_{i+1}$, respectively. 

\section{Spectrum of the position operators}

Note that for each fixed $m\geq 1$ the operators $x^{(m)}_i$ have the same spectrum.
We denote by $x^{(m)}$ any of them and by $\mu^{(m)}_t$ the spectral measure of $x^{(m)}$ associated with the vacuum vector, namely
the unique Borel measure $\mu_t^{(m)}$ on $\s(x^{(m)})$ such that 
$$\langle f(x^{(m)})\Om,\Om \rangle=\int_{\s(x^{(m)})} f {\rm d}\mu_t^{(m)}$$
for every $f\in C(\s(x^{(m)}))$. Let ${\rm supp}\mu_t^{(m)}$ be the support of
$\mu_t^{(m)}$, 
%This is known to be a finite set. Moreover, it is strictly contained in 
%$\s(x^{(m)})$ because $\Om$ is not a separating vector for $x^{(m)}$.
and denote by $r^{(m)}$ the position operator $x^{(m)}$ corresponding to
$t=1$.\\
At this point, it is worth recalling that $\mu_1^{(m)}$ is known explicitly
for every $m\geq 1$. Indeed, from Theorem 5.3  in \cite{FL} one has 
$$\mu_1^{(m)}=\sum_{k=1}^{m+1} b_{m, k}\delta_{z_{m,k}}\, ,$$
where $z_{m,k}= 2\cos \big(\frac{k\pi}{m+2}\big)$ and $b_{m, k}=\frac{2\sin^2[k\pi/(m+2)]}{m+2}$
for every $k=1, \ldots, m+1$.\\
To our knowledge, the general case has not been addressed elsewhere.

We need to recall a few basic facts. The Cauchy transform of a probability measure
on $\br$ is by definition the complex function
$$G_\mu(z)=\int_{-\infty}^{\infty} \frac{{\rm d}\mu(z)}{z-x}$$
defined for $z$ in the open upper  half-plane $\bc^+$.
If ${\rm supp}\mu$ is compact, then it is known (see \cite{Ak}) that  $G_\mu$ can be written as a convergent
continuous fraction of the type
\begin{equation*}
G_\mu(z)= \cfrac{1}{z- \a_0-\cfrac{\b_0}{z-\a_1-\cfrac{\b_1}{z-\a_2-\cfrac{\b_2}{\ddots}}}}
\end{equation*}
where $\{\a_n: n\in\bn\}$ and $\{\b_n: n\in\bn\}$ are the sequences of the so-called Jacobi parameters, see
\cite{Sz}. The $\a_n$'s are all zero when $\m$ is symmetric, and furthermore, by virtue of \cite[Theorem 4.1]{AB} in one-mode type interacting Fock spaces  the sequence $\{\b_n:n\in\bn\}$
 is given by  
$\b_0=1$, $\b_n=\frac{\l_n}{\l_{n-1}}$ for $n\geq 1$.\\
In addition, for each $i\in\bz$ and $n\in\bn$, by \eqref{creator} and \eqref{ann} one has that $\langle x_i^{2n+1}\Om, \Om\rangle=0$. Consequently, if we now denote by $\mu_t$ the spectral measure of any of the self-adjoint operators $x_i$ on the full
$t$-free Fock space $\cf_t(\ch)$ w.r.t. the vacuum, its Cauchy transform is nothing else than
\begin{equation*}
G_{\mu_t}(z)= \cfrac{1}{z- \cfrac{1}{z-\cfrac{t}{z-\cfrac{t}{\ddots}}}}\, ,
\end{equation*}
see also  \cite{BW}.
Since  for our measures $\mu_t^{(m)}$ the sequence $\{\b_n\}_{n\in\bn}$ is given by  
$\b_0=1$, $\b_n= t$ for
$n=1, 2, \ldots, m$ and $\b_n=0$ as soon as $n\geq m+1$, their
 Cauchy transforms $G_{\mu_t^{(m)}}$ 
can be written as  finite fractions as follows
$$G_{\mu_t^{(0)}}=\frac{1}{z}\,, G_{\mu_t^{(1)}}=\cfrac{1}{z-\cfrac{1}{z}}\,,G_{\mu_t^{(2)}}=\cfrac{1}{z-\cfrac{1}{z-\cfrac{t}{z}}}\,,
G_{\mu_t^{(3)}}=\cfrac{1}{z-\cfrac{1}{z-\cfrac{t}{z-\cfrac{t}{z}}}}$$
and so on.\\ 

In order to state our first result, we first need to establish some notation, for the atoms of the discrete measure $\mu_t^{(m)}$
are in fact the zeros of polynomial expressions in which the Chebyshev polynomials show up. In the sequel, we will denote by $U_m$, $m\geq 0$, the Chebyshev polynomals of the second kind, see \cite{Sz}. For completeness' sake, we recall that they are explicitly given by the formula
$$U_m(x)=\frac{\sin[(m+1)\arccos x]}{\sin[\arccos x]}\,, \quad x\in (-1, 1)\,.$$

\begin{prop}\label{finsupp}
For each $m\geq 1$, the measure $\mu_t^{(m)}$ is given by
$$\mu_t^{(m)}=\sum_{i=1}^{m+1} b_i\delta_{z_i}\,,$$
where $\{z_i: i=1, \ldots, m+1\}$ is the set  of the $(m+1)$ real zeros of the polynomial
$$\sqrt{t}zU_m\bigg(\frac{z}{2\sqrt{t}}\bigg)- U_{m-1}\bigg(\frac{z}{2\sqrt{t}}\bigg)\, ,$$ and

$$b_i= {\rm Res}_{z=z_i}\frac{\sqrt{t}U_m(\frac{z}{\sqrt{t}})}{U_m\bigg(\frac{z}{2\sqrt{t}}\bigg)- U_{m-1}\bigg(\frac{z}{2\sqrt{t}}\bigg)}$$
for $i=1, \ldots, m+1$.

\end{prop}

\begin{proof}
Throughout the proof by $x^{(m)}$ we will always mean $x^{(m)}_1$.
By looking at the formula of $G_{\mu_t^{(m)}}$ given above, one sees that  $G_{\mu_t^{(m)}}$ can also be written in terms  of $G_{\mu_1^{(m-1)}}$  as 
\begin{equation}\label{Simone}
G_{\mu_t^{(m)}}= \frac{1}{z-\frac{1}{\sqrt t} G_{\mu_1^{(m-1)}}\big(\frac{z}{\sqrt{t}}\big)}
\end{equation}
As is known, the atoms of a measure are the poles of its Cauchy transform. Now the poles of $G_{\mu_t^{(m)}}$ are the zeros
of $z-\frac{1}{\sqrt t} G_{\mu_1^{(m-1)}}\big(\frac{z}{\sqrt{t}}\big)$.
 From \cite[Lemma 5.2]{FL} we know that 
\begin{equation}\label{uwe}
G_{\mu_1^{(m-1)}}(z)=\frac{U_{m-1}(z/2)}   {U_m(z/2)}\,.
\end{equation}
Inserting \eqref{uwe} into \eqref{Simone} we arrive at
$$G_{\mu_t^{(m)}}=\frac{\sqrt{t}U_m\left(\frac{z}{2\sqrt{t}}\right)}{\sqrt{t}zU_m\bigg(\frac{z}{2\sqrt{t}}\bigg)- U_{m-1}\bigg(\frac{z}{2\sqrt{t}}\bigg)}\, .$$ 
Therefore, the poles of $G_{\mu_t^{(m)}}$ are the zeros of the polynomial $\sqrt{t}zU_m\bigg(\frac{z}{2\sqrt{t}}\bigg)- U_{m-1}\bigg(\frac{z}{2\sqrt{t}}\bigg)$.
  Define $V^{(m)}:={\rm span}\{\Om, e_1, e_1\otimes e_1, \ldots, e_1^{\otimes^m}\}$.
Note that $V^{(m)}$ is an invariant subspace for $x^{(m)}$ and $\Om$ is cyclic for the restriction $x^{(m)}\lceil_{V^{(m)}}$.
By cyclicity it follows that the eigenvalues of $x^{(m)}\lceil_{V^{(m)}}$ do not have multiplicity, that is 
$|\s(x^{(m)}\lceil_{V^{(m)}})|=m+1$. Since $\Om$ is also separating for
$x^{(m)}\lceil_{V^{(m)}}$, we have ${\rm supp}\mu_t^{(m)}=\s(x^{(m)}\lceil_{V^{(m)}})$.
As a result, the support of $\mu_t^{(m)}$ is finite and thus it is the same as the set of the zeros of $\sqrt{t}zU_m\bigg(\frac{z}{2\sqrt{t}}\bigg)- U_{m-1}\bigg(\frac{z}{2\sqrt{t}}\bigg)$.\\
Finally,  the formula of the weights of the atoms follows from applying \emph{e.g.} Proposition 1.104 in \cite{HO}.
\end{proof}

\begin{example}\label{firstvalues}
For the cases corresponding to the first values of $m$, the measure $\mu_t^{(m)}$ can be determined explicitly.
To begin with,  $\mu_t^{(1)}$ is clearly $\frac{1}{2}\delta_{-1}+\frac{1}{2}\delta_{1}$ for all $t>0$.
For $m=2$, one has 
$$\mu_t^{(2)}=\frac{1}{2\sqrt{t}(1+t)}\delta_{-\sqrt{1+t}}+\frac{\sqrt{t}}{t+1}\delta_{0}+\frac{1}{2\sqrt{t}(1+t)}\delta_{\sqrt{1+t}}\,. $$

\end{example}

The  next result provides a formula for the spectrum of $x^{(m)}$ in terms of the supports of $\mu^{(m)}_t$ and all
$\mu_1^{(k)}$  with $k=1, 2, \ldots, m-1$.

\begin{thm}\label{spectre}
For every real $t>0$ and integer $m\geq 1$, the spectrum of $x^{(m)}$ is given by
\begin{equation}\label{spectrum}
\s(x^{(m)})={\rm supp}\mu^{(m)}_t \bigcup \sqrt{t}(\cup_{k=1}^{m-1}{\rm supp}\mu_1^{(k)})\cup\{0\}\,
\end{equation}
\end{thm}

\begin{proof}
As in the previous proof, by $x^{(m)}$ we will always mean $x^{(m)}_1$.\\
We first decompose $\cf_t^{(m)}(\ch)$ into a direct sum
of suitable cyclic subspaces.
Let $V^{(m)}$ be the cyclic subspace generated by $\Om$ we introduced in the proof above.
As already pointed out, the spectrum of the restriction $x^{(m)}\lceil_{V^{(m)}}$ is equal to
${\rm supp}\mu_t^{(m)}$.\\
For every $j\neq 1$, let $V_j^{(m)}\subset\cf_t^{(m)}(\ch) $
be  the cyclic subspace generated by $e_j$. Note that 
each $V_j^{(m)}$ is the $m$-dimensional subspace ${\rm span}\{e_j, e_1\otimes e_j, \ldots, e_1^{\otimes^{m-1}}\otimes e_j\}$.
For any integer $2\leq k\leq m-1$, denote by $V_{j_1, j_2, \ldots, j_k}^{(m)}\subset\cf_t^{(m)}(\ch) $ the cyclic subspace generated by the vector $e_{j_1}\otimes\cdots e_{j_k}\otimes e_1$,  where $j_1\neq 1$
and $j_2, \ldots, j_k\in\bz$.
Note that each $V_{j_1, j_2, \ldots, j_k}^{(m)}$ is the $(m-k+1)$-dimensional subspace
${\rm span}\{e_{j_1}\otimes\cdots e_{j_k}\otimes e_1, e_{j_1}\otimes\cdots e_{j_k}\otimes e_1\otimes e_1, \ldots,e_{j_1}\otimes\cdots e_{j_k}\otimes e_1^{\otimes^{m-k}}  \}$.
It is easy to see that the following decomposition holds
$$\cf^{(m)}(\ch)= V^{(m)}\oplus(\oplus_{j\neq 1} V_j^{(m)})\oplus(\oplus_{j_1\neq 1} V_{j_1, j_2, \ldots, j_k}^{(m)})\,\, .$$
By direct computation one can verify the following unitary equivalences:
$$x^{(m)}\lceil_{V_j^{(m)}}\cong \sqrt{t}\,\, r^{(m-1)}\lceil_{V_1^{(m-1)}}$$
$$x^{(m)}\lceil_{ V_{j_1, j_2, \ldots, j_k}^{(m)}}\cong\sqrt{t}\,\, r^{(m-k-1)}\lceil_{V_1^{(m-k-1)}}\,.$$
with $r^{(m)}$ being the position operator $x^{(m)}$ for $t=1$.
Therefore, we find
$$\s(x^{(m)})={\rm supp}{\mu_t^{(m)}}\bigcup\sqrt{t}\big(\cup_{k=0}^{m-1}{\rm supp}\mu_1^{(m-k-1)}\big)\cup\{0\}\,.$$
\end{proof}

 An entirely explicit formula can be given for the spectrum of the position operator for $t=1$.
\begin{cor}
For each $m\geq1$, the spectrum of $r^{(m)}$ is 
$$\sigma(r^{(m)})=\bigcup_{k=0}^{m}\bigg\{2\cos \left(\frac{l\pi}{m+2}\right): l=0, 1, \ldots, k\bigg\}\, .$$
\end{cor}

Going back to the case of a general $t>0$, we next show that the support of the spectral mesaure is always
properly contained in the spectrum of the position operator.
\begin{cor}\label{strictinc}
For every real $t>0$ and  integer $m\geq 1$, the support of $\mu_t^{(m)}$ is strictly contained in $\s(x^{(m)})$.
\end{cor}

\begin{proof}
Thanks to Equality \eqref{spectre}, it is enough to show that 
$$\big|\cup_{k=1}^{m-1}{\rm supp}\mu_1^{(k)}\big| > \big|{\rm supp}\mu^{(m)}_t \big|= m+1\, ,$$
where $|\cdot|$ is the cardinality of a (finite) set.
Now $\big|\cup_{k=1}^{m-1}{\rm supp}\mu_1^{(k)}\big| \geq \big|{\rm supp}\mu_1^{(m-1)}\cup{\rm supp}\mu_1^{(m-2)}\big|$.
Because ${\rm supp}\mu_1^{(m-1)}$ and ${\rm supp}\mu_1^{(m-2)}$ are disjoint sets, see {\it e.g.} Section 5 of  \cite{FL}, we 
find $\big|\cup_{k=1}^{m-1}{\rm supp}\mu_1^{(k)}\big| \geq m+m-1= 2m-1$. Since  $2m-1>m+1$ as soon as 
$m>2$, the conclusion follows for any such $m$. Finally, the cases $m=1$ and $m=2$ have already been handled in Example \ref{firstvalues} by explict computation.

\end{proof}
\begin{rem}
In particular, from Corollary \ref{strictinc} we see that for every $m\geq 1$ the vacuum vector $\Omega$ is not separating for the $C^*$-algebra generated by $x^{(m)}$.
\end{rem}

In a similar fashion we can determine the spectrum of the position operators $x_i$ on the full $t$-free Fock space.
As usual, we will drop the subscripts and just write $x$ to refer to any of the position operators.
\begin{prop}\label{spectre2}
For every real $t>0$, the spectrum of $x$ is given by
\begin{equation}\label{spectrum}
\s(x)={\rm supp}\mu_t \,,
\end{equation}
with ${\rm supp}\mu_t$ being $[-2\sqrt{t}, 2 \sqrt{t}]$ for $t\geq\frac{1}{2}$
and  ${\rm supp}\mu_t=[-2\sqrt{t}, 2 \sqrt{t}]\cup \{\pm 1/\sqrt{1-t}\}$
for $t<\frac{1}{2}$.
In particular, for every $t>0$ the vacuum vector is separating for $x$.
\end{prop}
\begin{proof}
Arguing as in the proof of Theorem \ref{spectre}, one finds that 
$\s(x)= {\rm supp}\mu_t\cup \sqrt{t}\,{\rm supp}\mu_1$.
Now $\sqrt{t}\, {\rm supp}\mu_1$ is $[-2\sqrt{t}, 2\sqrt{t}]$ since $\mu_1$ is the standard Wigner distribution, whereas 
 ${\rm supp}\mu_t$ is  $[-2\sqrt{t}, 2\sqrt{t}]\cup\{\pm\frac{1}{\sqrt{1-t}}\}$  for $t<\frac{1}{2}$ or 
$[-2\sqrt{t}, 2\sqrt{t}]$ for $t\geq\frac{1}{2}$, as shown in  \cite{BW, R}.
In either case, $\sqrt{t}{\rm supp}\mu_1$ is contained in ${\rm supp}\mu_t$, hence
$\s(x)={\rm supp}\mu_t$.
\end{proof}

\section{Ergodic properties of the shift}

We denote by $\mathcal{A}^{(m)}_t$ the dense $*$-subalgebra of $\ga^{(m)}_t$ generated by the
set $\{a_i^{(m)}: i\in\bz\}$. In the following, to ease the notation we will write
$a_i$ instead of $a_i^{(m)}$ when $m$ is fixed. 
We will denote by $P_{\ch^{\otimes^h}}$, $h=0, 1, \ldots, m$, the projection
of $\cf_t^{(m)}(\ch)$ onto $\ch^{\otimes^h}$, where  $\ch^{\otimes^0}$ is understood
as $\bc\Om$.
\begin{lem}\label{wick}
For every $t>0$ and every integer $m\geq 1$, the projections $P_{\ch^{\otimes^h}}$ belong
to $\ga^{(m)}_t$ for $h=0, 1, \ldots, m$. Moreover,
the $*$-algebra $\mathcal{A}^{(m)}_t$ is contained in the linear span of all monomials of the type
$a^\dag_{i_1}\cdots a^\dag_{i_r}a_{j_1}\cdots a_{j_l}P_{\ch^{\otimes^h}}$, 
with $i_1,\ldots, i_r$,  $j_1,\ldots, j_l\in\bz$, $r, l=0, 1, \ldots, m$ and 
 $h=0, 1, 2, \ldots, m-1$.
\end{lem}

\begin{proof}
We start by noting that for all $i\in\bz$ and $h=0, 1, \ldots, m-1$ the following equalities hold as a straightforward consequence of \eqref{creator}--\eqref{ann}:
\begin{align}
\begin{split}\label{commproj}
&a_ia^\dag_j=\delta_{i,j}\bigg(P_\Om+ t\sum_{h=1}^{m-1}P_{\ch^{\otimes^h}}\bigg)\,,\\
&a_iP_\Om=0=P_\Om a^\dag_i\,,\\
&P_{\ch^{\otimes^h}}a_i=a_iP_{\ch^{\otimes^{h+1}}}\,,\\
&P_{\ch^{\otimes^{h+1}}} a^\dag_i=a^\dag_i P_{\ch^{\otimes^h}}\,.
\end{split}
\end{align}
 Thi first part of the statement can be proved by induction on $m$.
The basis, that is $m=1$, follows trivially from the first of Relations \eqref{commproj}.
 For the inductive step,  the first thing to observe is 
that $\sum_{h=0}^{m}P_{\ch^{\otimes^h}}$ is in $\ga^{(m+1)}_t$ due to
the first of Relations \eqref{commproj}.
Suppose that  $P_\Om, P_\ch, \ldots, P_{\ch^{\otimes^{m}}}$ lie in 
$\ga^{(m)}_t$. Thanks to the equality 
$a_i^{(m)}=  (I-P_{\ch^{\otimes^{m+1}}})a_i^{(m+1)}(I-P_{\ch^{\otimes^{m+1}}})$, which is easily verified, the above projections are in $\ga^{m+1}_t$ as well and we are done.\\
For the second part of the statement, from \eqref{commproj}
 it is easy to see that any monomial in the generators of $\mathcal{A}^{(m)}_t$
can be rewritten as a finite linear combination of words as in the statement.
Indeed, if  a word $a^\sharp_{k_1}a^\sharp_{k_2}\cdots a^\sharp_{k_s}$,
where $\sharp$ is either $1$ or $\dag$, is not in the desired order, then it must display
a factor of the type $a_ia_j^\dag$.  But this is reduced by the first quality above.
Applying the remaining equalities as many times as necessary, the sought
form is finally arrived at. \\
\end{proof}

We denote by $\tau$ the shift automorphism on $\ga^{(m)}_t$, that is the automorphism completely
determined by $\tau(a_i)=a_{i+1}$, for every $i\in\bz$. Clearly,
$\tau$ restricts to an autmorphism of $\mathfrak{R}_m$, which we continue
to denote by $\tau$. Note that $\tau(x_i^{(m)})=x_{i+1}^{(m)}$, for
every $i\in\bz$.\\

In the following lemma we prove an anologue of both estimates in \cite[Proposition 3.2 ]{DyF}
and  in \cite[Proposition 4.2]{CFL} adapted to the
$t$-free Fock space. To avoid any ambiguity, we denote by $l^\dag$
 the creators  on full the $t$-free Fock space.
 Furthermore, we denote by $\sharp$ either $1$ or $\dag$.
 
\begin{lem}\label{estimatefull}
For every $t>0$, one has 
$$\left\|\sum_{k=0}^{n-1} \t^k(l^\dag_{i_1} l^\sharp_{i_2}\cdots l^\sharp_{i_r}l^\sharp_{j_1}\cdots l^\sharp_{j_s})\right\|\leq  \sqrt{nt}  \max\{1, \sqrt{t}\}^{(s+r-1)}$$
for all  $r, s\in\bn$ and $i_1, i_2, \ldots i_r, j_1, j_2, \ldots, j_s\in\bz$.
\end{lem}

\begin{proof}
Let $\xi$ be a unit vector in $\ch^{\otimes^m}$. For $m\geq s$, one has
\begin{align*}
\left\|\sum_{k=0}^{n-1} \t^k(l^\dag_{i_1} l^\sharp_{i_2}\cdots l^\sharp_{i_r}l^\sharp_{j_1}\cdots l^\sharp_{j_s})\xi\right\|^2=&
\left\| \sum_{k=0}^{n-1} l^\dag_{i_1+k} l^\sharp_{i_2+k}\cdots l^\sharp_{i_r+k}l^\sharp_{j_1+k}\cdots l^\sharp_{j_s+k}\xi   \right\|^2\\
=& \left\|\sum_{k=0}^{n-1} l^\dag_{i_1+k} \xi_k  \right\|^2
\end{align*}
with $\xi_k:=l^\sharp_{i_2+k}\cdots l^\sharp_{i_r+k}l^\sharp_{j_1+k}\cdots l^\sharp_{j_s+k}\xi $ for $k=0, 1, \ldots, n-1$. Now note that
$\xi_k\in \ch^{\otimes^{m+r-s-1}}$ and $\|\xi_k\|^2\leq \max\{1, \sqrt{t}\}^{2(s+r-1)} $
due to $\|l_i\|\leq \max\{1, \sqrt{t}\}$ for every $i\in\bz$. But then
\begin{align*}
\left\|\sum_{k=0}^{n-1} l^\dag_{i_1+k} \xi_k  \right\|^2=& \bigg\langle \sum_{k=0}^{n-1} l^\dag_{i_1+k} \xi_k , \sum_{k'=0}^{n-1} l^\dag_{i_1+k'} \xi_{k'} \bigg\rangle\\
=& \bigg\langle \sum_{k=0}^{n-1} e_{i_1+k}\otimes \xi_k , \sum_{k'=0}^{n-1} e_{i_1+k'}\otimes \xi_{k'} \bigg\rangle\\
=&\sum_{k=0}^{n-1} \|e_{i_1+k}\otimes \xi_k \|^2= \sum_{k=0}^{n-1}t \|\xi_k\|^2\\
\leq& nt \max\{1, \sqrt{t}\}^{2(s+r-1)}\, .
\end{align*}

\end{proof}

A version of the lemma above tailored to the case of the truncated $t$-free spaces can be proved as well.
For convenience, we single out the relative result below. The proof, though, is left out as it is a very minor variation of the proof
above.

\begin{cor}\label{estimatetrunc}
For every $t>0$ and every integer $m\geq 1$, one has 
$$\left\|\sum_{k=0}^{n-1} \t^k(a^\dag_{i_1}\cdots a^\dag_{i_r}a_{j_1}\cdots a_{j_s})\right\|\leq \sqrt{nt}\max\{1, \sqrt{t}\}^{2m-1}$$
for all  $r, s\in\bn$ and $i_1, i_2, \ldots i_r, j_1, j_2, \ldots, j_s\in\bz$.
\end{cor}

With the above lemmas at hand, we are ready
to prove that the shift is uniquely ergodic with respect to its fixed-point subalgebra. For completeness, we recall that unique ergodicity w.r.t. the fixed-point subalgebra is a generalization of unique ergodicity due to
Abadie and Dykema \cite{AD}.  Before going on, let us recall a few definitions. By a $C^*$-dynamical system we mean a pair $(\ga, \Phi)$, where $\ga$ is a $C^*$-algebra and $\Phi$ a $*$-automorphism. The fixed-point subalgebra of $(\ga, \Phi)$ is the $C^*$-subalgebra $\ga^\Phi:=\{a\in \ga : \Phi(a)=a\}$. A $C^*$-dynamical system is said to be uniquely ergodic
w.r.t. the fixed-point subalgebra when for every $a\in\ga$ the Ces\`{a}ro averages $\frac{1}{n}\sum_{k=0}^{n-1}\Phi^k(a)$ converge in norm to some 
(fixed) element of $\ga$. Clearly, this notion reduces to unique ergodicity when the fixed-point subalgebra is trivial. Furthermore, in the aforementioned
paper several equivalent conditions  are provided for a dynamical system to be uniquely ergodic w.r.t. the fixed-point subalgebra.
One  particularly relevant to the present work is that a dynamical system enjoys this property if and only if
every state on the fixed-point subalgebra has precisely one invariant extension to the whole $C^*$-algebra, see Theorem 3.2 in
\cite{AD}.

\begin{prop}\label{uniquerg}
For every $t>0$ and integer $m\geq 1$, the $C^*$-dynamical system $( \ga^{(m)}_t, \tau)$ is uniquely
ergodic w.r.t the fixed-point subalgebra.
\end{prop}

\begin{proof}
We will prove the statement by showing that, for every $a\in \ga^{(m)}_t$, the Ces\`{a}ro averages
$\frac{1}{n}\sum_{k=0}^{n-1}\tau^k(a)$ converge in norm.\\
Since  $\big\|\frac{1}{n}\sum_{k=0}^{n-1}\t^k\big\|\leq 1$ for every $n\in\bn$, thanks to Lemma \ref{wick} it is enough
to prove convergence only for elements of the form
$$a^\dag_{i_1}\cdots a^\dag_{i_r}a_{j_1}\cdots a_{j_l}P_{\ch^{\otimes^h}}$$
for $r, l\in\bn$, $i_1, \ldots, i_r, j_1, \ldots, j_l\in\bz$, and $h=0, 1, \ldots m$.\\
Furthermore, only  words with $r+l\geq1$ need to be taken care of.
For words of this type, the sequence  $\frac{1}{n}\sum_{k=0}^{n-1}\tau^k(a^\dag_{i_1}\cdots a^\dag_{i_r}a_{j_1}\cdots a_{j_l}P_{\ch^{\otimes^h}})$
tends to $0$ in norm by Corollary \ref{estimatetrunc}.\\
%The case of the full $t$-free Fock space can be handled in the same way using the estimates in Lemma \ref{estimatefull}.
\end{proof}

Unique ergodicity can now be taken advantage of to determine what the fixed-point subalgebra is like.

\begin{prop}\label{fixedpoint}
For every $t>0$ and every integer  $m\geq 1$, the subalgebra of $\t$-invariant elements of $\ga^{(m)}_t$
is ${\rm span}\{P_\Om, P_{\ch}, \ldots, P_{\ch^{\otimes^m}}\}$.
\end{prop}

\begin{proof}
Thanks to Proposition \ref{uniquerg},  the formula $\displaystyle {E:=\lim_{n\rightarrow\infty}\frac{1}{n}\sum_{k=0}^{n-1}\t^k}$ defines a conditional expectation onto the fixed-point subalgebra.
Now the argument employed in the proof of the proposition above shows that the latter is linearly spanned by the projections $P_{\ch^{\otimes^h}}$, $h=0, 1, \ldots, m$.
\end{proof}

We are now in a position to describe the structure of the $*$-weakly compact convex set of all shift-invariant states on $\ga^{(m)}_t$, which we denote by $\cs_{\bz}(\ga^{(m)}_t)$.

\begin{prop}\label{simplex}
For every $t>0$ and every  $m\geq 1$, $\cs_{\bz}(\ga^{(m)}_t)$ is a finite-dimensional
(Choquet) simplex with $m+1$ extreme states $\om_i$, $i=0, 1, \ldots, m$, uniquely determined by
$\om_i(P_{\ch^{\otimes^j}})=\delta_{i, j}$, for $i,j=0, 1, \ldots, m$.
\end{prop}

\begin{proof}
By unique ergodicity w.r.t. the fixed-point subalgebra, $(\ga^{(m)}_t)^\tau$, the restriction map
$\cs_{\bz}({\ga^{(m)}_t})\ni\om\rightarrow \om\lceil_{(\ga^{(m)}_t)^\tau}\in\cs((\ga^{(m)}_t)^\tau)$
is an affine homeomorphism between compact convex sets. The conclusion now follows from
Proposition \ref{fixedpoint} because $(\ga^{(m)}_t)^\tau\cong \bc^{m+1}$ (as $C^*$-algebras).

\end{proof}

The case of the full $t$-free Fock space looks different insofar as the fixed-point subalgebra is two-dimensional for $t\neq 1$.%$vacuum state  and the
%state at infinity are now the only  extreme shift-invariant states for $t\neq 1$. More precisely, we have the following
%result.
\begin{prop}\label{fullue}
For every $t>0$, $(\ga_t, \tau)$ is uniquely ergodic w.r.t. the fixed-point subalgebra.
Moreover, for $t=1$ the fixed-point  subalgebra is trivial, whereas it is
${\rm span}\{P_\Om, I-P_\Om\}$ for $t\neq 1$.
\end{prop}
\begin{proof}
We need only deal with $t\neq 1$ since the free case, {\it i.e.} $t=1$, is known, see \cite[Theorem 3.3]{DyF}.
From the relations
$l_iP_\Om=0$ (hence $l_i(I-P_\Om)= l_i$) and
$l_i l^\dag_j=\delta_{i, j}(P_\Om+t (I-P_\Om))$, $i, j\in\bz$, one sees that any word in the 
$t$-free creators and annihilators can be written as a linear combinations of 
$P_\Om, (I-P_\Om), w, wP_\Om$ and their adjoints, where $w$ is a word (possibly of length $1$) either starting with a creator
or ending in an annihilator. Applying Lemma \ref{estimatefull} to $w$, one finally finds
that the range of the conditional expectation onto the fixed-point subalgebra 
is spanned by $P_\Om$ and $I-P_\Om$.
\end{proof}

\begin{rem}
As an immediate consequence of Proposition \ref{fullue}, we have that $\cs_{\bz}(\ga_t)$ has exactly two extreme states when $t\neq 1$.%, $\om_i$, $i=0, 1$, uniquely determined by
%$\om_i(P_{\ch^{\otimes^j}})=\delta_{i, j}$, for $i,j=0, 1, \ldots, m$.
\end{rem}

The rest of the section is devoted to the $C^*$-subalgebra generated by the position operators, that is
$\mathfrak{R}^{(m)}_t$. In particular, we are going to show that 
$\mathfrak{R}^{(m)}_t$ is actually the same as $\ga^{(m)}_t$. To this end, we start by ascertaining that all projections
onto the $k$-particle spaces lie in $\mathfrak{R}^{(m)}_t$.
\begin{lem}\label{proj}
For every $t>0$ and for every integer $m\geq 1$, the set of projections $\{P_{\ch^{\otimes^h}}: h=0, 1, \ldots, m\}$ sits in 
$\mathfrak{R}^{(m)}_t$. 
\end{lem}

\begin{proof}
An easy adapatation of Lemma 2.2 in \cite{W} shows that the sequence
$$\frac{1}{2n+1}\sum_{k=-n}^n (x_i^{(m)})^2$$
converges strongly to $P_\Om+t(P_\ch+\cdots P_{\ch^{\otimes^{m-1}}})$.
By virtue of Proposition \ref{uniquerg} the convergence of the sequence is actually in norm, which means the
projection $P_\Om+t(P_\ch+\cdots P_{\ch^{\otimes^{m-1}}})$ belongs to $\mathfrak{R}^{(m)}_t$.
In particular,  $P_\Om$ and $P_\ch+\cdots P_{\ch^{\otimes^{m-1}}}$ lie in  $\mathfrak{R}^{(m)}_t$.
The conclusion is then reached by induction on $m$ as was done in the proof of
Lemma \ref{wick}.

\end{proof}

\begin{prop}
For each $m\geq 1$, one has $\mathfrak{R}^{(m)}_t=\ga^{(m)}_t$.
\end{prop}

\begin{proof}
All we have to prove is that, for every $i\in\bz$, the operator $a_i^{(m)}$ sits in  $\mathfrak{R}^{(m)}_t$.
This is a straightforward consequence of Lemma \ref{proj} and the equality 
$$a_i^{(m)}=\sum_{k=0}^{m-1} P_{\ch^{\otimes^k}} x_i^{(m)} P_{\ch^{\otimes^{k+1}}}\, ,$$
which can be checked by direct computation.
\end{proof}
We end the section by showing that  $\ga^{(m)}_t$ is an irreducible
$C^*$-subalgebra of $\cb(\cf^{(m)}_t(\ch))$. This was already known in the full case, \cite{W}.

\begin{prop}\label{tirreducible}
For every $t>0$ and for every integer $m\geq 1$, the $C^*$-algebra  $\mathfrak{R}^{(m)}_t=\ga^{(m)}_t$ acts irreducibly on $\cf_t^{(m)}(\ch)$.
\end{prop}

\begin{proof}
We will show that the commutant $(\ga^{(m)}_t)'$ is trivial. If $T\in\cb(\cf_t^{(m)}(\ch))$ is in  $(\ga^{(m)}_t)'$, then in particular
$TP_\Om=P_\Om T$, hence $T\Om=\lambda \Om$, for some $\lambda\in\bc$.
Since $\Om$ is cyclic for $\ga^{(m)}_t$, $\Om$ is separating for $(\ga^{(m)}_t)'$, which means $T=\lambda I$.
\end{proof}

\section*{Acknowledgments}
\noindent
We acknowledge  the support of Italian INDAM-GNAMPA.

\end{document}